\documentclass[a4paper, english, 10pt]{amsart}

\usepackage{amsfonts}
\usepackage{caption}
\usepackage{forest}
\usepackage{amsmath}
\usepackage{amssymb}
\usepackage{tikz}
\usepackage[cp850]{inputenc}
\usepackage[bookmarksnumbered,plainpages]{hyperref}
\usepackage{color}
\usepackage{mathtools}
\usepackage{MnSymbol}
\usepackage{subfigure}
\usepackage[all]{xy}
\allowdisplaybreaks
\newtheorem{theorem}{Theorem}[section]

\newtheorem{proposition}[theorem]{Proposition}
\newtheorem{corollary}[theorem]{Corollary}

\newtheorem{remark}[theorem]{Remark}
\newtheorem{lemma}[theorem]{Lemma}

\newtheorem{definition}[theorem]{Definition}
\newtheorem{example}[theorem]{Example}

\def\sym#1{\mathrm{Sym}(#1)}

\def\aut#1{\mathrm{Aut}(#1)}
\def\End#1{\mathrm{End}(#1)}
\def\aff#1{\mathrm{Aff}#1}

\def\S{S}

\def\RMlt{\mathrm{RMlt}}

\def\setof#1#2{\{#1\, : \,#2\}}

\newcommand*\xbar[1]{%
   \hbox{%
     \vbox{%
       \hrule height 0.5pt 
       \kern0.5ex
       \hbox{%
         \kern-0.1em
         \ensuremath{#1}%
         \kern-0.1em
       }%
     }%
   }%
}

\makeindex 
\begin{document}

\markboth{Authors' Names}{Instructions for Typesetting Manuscripts using \LaTeX}


\title{On the axioms of singquandles}

\author{M. Bonatto}

\address{Department of Mathematics and Computer Science, University of Ferrara, Via Macchiavelli 30, 44121 Ferrara, Italy\\ marco.bonatto.87@gmail.com}

\author{A Cattabriga}

\address{Department of Mathematics, University of Bologna, Piazza di Porta San Donato 5, 40126 Bologna, Italy\\ alessia.cattabriga@unibo.it}



\begin{abstract}
In this paper we deal with the notion of  singquandles introduced in \cite{SingInvol}. This is an algebraic structure that naturally axiomatizes   Reidemeister moves for singular links,  similarly to what happens for ordinary links and quandle structure.  We present a new axiomatization that  shows different algebraic aspects and simplifies applications.  We also reformulate and simplify  the axioms for affine singquandles (in particular in the idempotent case). 
\end{abstract}
\maketitle

\section{Introduction}

Singular knot theory was introduced in 1990 by Vassiliev \cite{Va} as an extension of classical knot theory allowing also immersions of $S^1\to S^3$ with  singularities; the aim was to get informations on knots  by studying the space of all their isotopy classes:  singular knots gave rise to  a decreasing filtration on the infinite dimensional vector space generated by isotopy classes of knots. Together with the introduction of this extension, the notion of finite type (or Vassiliev) invariants, as invariants vanishing on some step of this filtration, was introduced  producing  a new point of view on knot theory. Since then, many knot invariants, as well as different knot representations and techniques have been extended to singular knots and links (see  for example \cite{Bir,Fie}). Recently, in \cite{SingInvol}, a singular link invariant having the form of a  binary  algebraic structure and  called singquandle, was defined; as the name suggests, this structure extends to the singular case the quandle invariant for classical links. Quandles, or distributive grupoids, were introduced in the 1980s by Joyce \cite{J} and, independently, by Matveev \cite{Matveev}: the fundamental quandle $Q(L)$ of a link $L$, axiomatizes the Reidemeister moves and is a  classifying invariant for prime knots. Even if  comparing quandles is as difficult as comparing links, as for the case of Vassiliev theory,  the introduction of quandles (and racks) in knot theory paved the way for the construction of  new  invariants and tecniques.  Moreover, beside the interested of quandles for knot theory, these structures are relevant in many other areas, as  theoretical physics, for the study of the Yang-Baxter equation  (see \cite{AG, ESS, EGS}) or   abstract algebra itself (see \cite{Stanos, CP, Maltsev_paper}). In \cite{Generating_Reidemeister} and \cite{OrientedSing} the singquandle construction is done for the oriented  case, while  in  \cite{Psy}  the notion of psyquandles is introduced for the case of pseudoknots and singular knots and links as a generalization of biquandle structures for classical and virtual links \cite{biquandle}.\\

In this paper we deal with algebraic structures associated to singular links. More precisely,  we reformulate the definition of oriented and non-oriented singquandles, by using the language of binary operation: we  simplify the axioms introduced in the above mentioned papers and we  prove  the independence of our axioms.  With this new definition  we are able to  prove some algebraic properties of these structures and  to simplify some associated constructions as, for example, Alexander singquandles introduced in  \cite{SingInvol}.\\

Starting from our new reformulation  of the oriented singquandles $SQ(L)$ associated to a singular link $L$, we  remark that, if $\bar L_+$ (resp. $\bar L_-$) is the link obtained by replacing each singular crossing of $L$  with a  positive (resp. negative)  crossing, see the left (resp. right part)  of Figure \ref{fig:es2}, then $Q(\bar L_+)$ (resp.  $Q(\bar L_-)$) is a quotient of $SQ(L)$. This could be a starting point to explore the connection between $SQ(L)$ and $\{Q(\bar L_i)\}$, with $\{\bar L_i\}$, being  the set of all regular links obtained from $L$ by replacing a singular crossing with either a positive or negative crossing.

The definition of the (oriented) singquandle associated to a singular link is purely combinatorial. We wonder if there is also a topological construction for such an object as for the fundamental quandle for classical links \cite{J, Matveev}. Alternatively, if there is a topological construction of a proper quotient of the fundamental (oriented) sinquandle as for the  fundamental biquandle of a link \cite{Eva}.

In Section \ref{preliminary} we recall  all the algebraic notion that will be used in the rest of the paper, as well as the definition of classical and singular  links and their associated algebraic structures. In Section \ref{oriented} we analyze the oriented case, reformulating the definition of oriented singquandle, while the non-oriented case is studied in Section \ref{non-oriented}. We conclude the paper by analyzing the case of affine and Alexander singquandles.

In the paper we sometimes use the software \cite{Prover9} to generate examples and non examples of binary algebraic structures.

\section{Preliminary results}
\label{preliminary}

\subsection{Binary structures and right quasigroups}

A binary operation $\cdot$ on a set $X$ is a mapping 
$$\cdot:X\times X\longrightarrow X,\quad (x,y)\mapsto x\cdot y$$
and a binary algebraic structure is a set $X$ endowed with a set of binary operations. Let $(X,\cdot)$ be such a structure, the right multiplication by $x\in X$ is the map defined by setting $$R_x:y\mapsto y\cdot x$$
and the {\it squaring mapping} as
$$\sigma: X\longrightarrow X,\quad x\mapsto x\cdot x.$$
A (bijective) map $f:X\longrightarrow X$ is said to be an endomorphism (automorphism) of $(X,\cdot)$ if $f(x\cdot y)=f(x)\cdot f(y)$ for every $x,y\in X$. The group of automorphism of $(X,\cdot)$ is denoted by $\aut{X,\cdot}$.

The structure $(X,\cdot)$ is a {\it right quasigroup} if $R_x$ is a permutation for every $a\in A$. Clearly we can define the right division associated to $\cdot$ as $x/y=R_y^{-1}(x)$. Thus, for the scope of this paper, a right quasigroup can be alternatively be defined as a binary algebraic structure $(X,\cdot,/)$ such that
$$(x\cdot y)/y=x=(x/y)\cdot y.$$
Note that also $(X,/,\cdot)$ is a right quasigroup. 
A right quasigroup $(X,\cdot,/)$ is said to be:
\begin{itemize}
\item[(i)] {\it permutation} if $x\cdot y=x\cdot z$ holds;
\item[(ii)] {\it idempotent} if $x\cdot x=x$ holds;
\item[(iii)] {\it involutory} if $(x\cdot y)\cdot y=x$ holds;
\item[(iv)] {\it right distributive} if $(x\cdot y)\cdot z=(x\cdot z)\cdot(y\cdot z)$ holds;
\item[(v)] {\it $2$-divisible} if $\sigma$ is bijective.
\end{itemize}
Idempotent permutation right quasigroups satisfy the identity $x\cdot y=x\cdot x=x$ and they are called {\it projection}. Idempotent right distributive right quasigroups are called (right) {\it quandle}.

Given a right quasigroup $(X,\cdot)$, the orbits with respect to the natural action of the group $\RMlt(X,\cdot)=\langle R_x,\, x\in X\rangle$ are called the {\it connected components} of $(X,\cdot)$ and we say that $(X,\cdot)$ is {\it connected} if such group is transitive on $X$. Note that $X$ is the union of the connected components of its generators, therefore the following statement follows.

\begin{lemma}\label{connected by gen}
Let $(X,\cdot)$ be a right quasigroup generated by $S\subseteq X$. Then $(X,\cdot)$ is connected if and only if the element of $S$ are in the same connected component.
\end{lemma}

Let $(A,+)$ be an abelian group, $f\in \aut{A,+}$, $g\in \End{A,+}$ and $c\in A$. The right quasigroup $(A,\cdot)$ defined by setting
$$x\cdot y=f(x)+g(y)+c$$
is called {\it affine} right quasigroup over $A$. We denote such right quasigroup by $\aff{(A,f,g,c)}$.

In the paper we usually deal with algebraic structure with two binary operations denoted by $(X,\cdot,*)$. In the sequel, we denote the operation $\cdot$ just by juxtaposition and the right multiplication mappings by $x\in X$ respectively by $R_x:y\mapsto y\cdot x$ and by $\rho_x:y\mapsto y*x$.

\subsection{$\S$-right quasigroups}

Let us introduce a class of right quasigroups that will be relevant in the present paper in connection with coloring invariants of singular knots. The right quasigroups satisfying the identity
\begin{equation}
x/y=x(yx)\label{reduced axioms} \tag{S}
\end{equation}
will be called {\it $\S$-right quasigroups}. 

\begin{lemma}\label{PropS0}
Let $(X,\cdot)$ be a binary algebraic structure. The following are equivalent:
\begin{itemize}
 \item[(i)] $(X,\cdot)$ is a $\S$-right quasigroup.
\item[(ii)] The identity 
\begin{equation}\label{S01}
(yx)(x(yx))=y
\end{equation}
 holds.
\item[(iii)] The identity 
\begin{equation}\label{S02}
(x(yx))y=x
\end{equation}
 holds.
\end{itemize}
\end{lemma}

\begin{proof}
(i) $\Rightarrow$ (ii) We have
$$y=(yx)/x\overset{\eqref{reduced axioms}}{=}(yx)(x(yx)).$$

(ii) $\Rightarrow$ (iii) Using the identity \eqref{S01} twice, we have
$$(x(yx))y=(x(yx))((yx)(x(yx)))=x.$$

(iii) $\Rightarrow$ (ii) Using the identity \eqref{S02} twice, we have
$$(yx)(x(yx))=(y((x(yx))y))(x(yx))=y.$$

(ii), (iii) $\Rightarrow$ (i) Let us denote $x\bullet y=x(yx)$. The identities \eqref{S01} and \eqref{S02} reads
$$(xy)\bullet y=(x\bullet y)y=x,$$
therefore $\bullet$ is the right division with respect to $\cdot$.
\end{proof}

It is easy to prove that the identity \eqref{reduced axioms} is equivalent to
\begin{align}
xy&=x/(y/x)\label{reduced axioms2}\tag{S'}.
\end{align}
Indeed, it is enough to replace $y$ by $y/x$ in order to get the identity \eqref{reduced axioms2} from \eqref{reduced axioms} and replace $y$ by $yx$ conversely. 
\begin{proposition}\label{S and S'}
Let $(X,\cdot,/)$ be a right quasigroup. The following are equivalent:
\begin{itemize}
\item[(i)] The map $$\varphi:X\times X\longrightarrow X\times X,\quad (x,y)\mapsto(xy,y/x)$$
 is an involution.
\item[(ii)] $(X,\cdot, /)$ is a $S$-right quasigroup.
\item[(iii)] $(X, /,\cdot)$ is a $S$-right quasigroup.
\end{itemize}
\end{proposition}

\begin{proof}
The items (ii) and (iii) are equivalent because of the previous remark.

(i) $\Leftrightarrow$ (ii), (iii) Since
$$\varphi^2(x,y)=\varphi(xy,y/x)=((xy)(y/x),(y/x)/(xy))$$
then $\varphi$ is an involution if and only if
\begin{align*}
(xy)(y/x)=x,\qquad (y/x)/(xy)=y,
\end{align*}
hold. Namely \eqref{reduced axioms} and \eqref{reduced axioms2} hold.
\end{proof}

Let $t=s$ be an identity that follows from \eqref{reduced axioms}. According to Proposition \ref{S and S'}, then also the identity $t'=s'$ where $\cdot$ and $/$ are interchanged follows. For instance the identity $x=(x(y((x(yx))y)))y$ is a consequence of \eqref{reduced axioms} and so also $x=(x/(y/((x/(y/x))/y)))/y$ holds for $\S$-right quasigroups.
 
Let us show some examples of $\S$-right quasigroups.
\begin{example}\label{Example of S}

\begin{itemize}

\item[(i)] Let $(X,\cdot)$ be an involutory right quasigroup. Then $(X,\cdot)$ is a $S$-right quasigroup if and only if 
\begin{align}\label{S for involutory}
x(yx)=xy
\end{align}
 holds. Thus, if $R_x=R_y$ whenever $x$ and $y$ are in the same connected component of $(X,\cdot)$ the identity \eqref{S for involutory} is satisfied. For instance, involutory permutation right quasigroups and $2$-reductive involutory (right) quandles have such property (see \cite{Medial} for the construction of such quandles).

\item[(ii)] Let $X=\aff(A,f,g,c)$. Then $X$ is an $S$-right quasigroup if and only if

\begin{equation} fg^2+f^2-1 =g+fgf=(f+fg+1)(c)=0.\label{affine S}
\end{equation}

Then $(A,+)$ has a $\mathbb{Z}[t,t^{-1},u]/(tut+u,t^2+tu^2-1)$-module structure. On the other hand, given a module $M$ over such ring, $(M,\cdot)$ where
$$x\cdot y =t x+u y+c$$
 is a $S$-right quasigroup if and only if $(1+t+tu)c=0$ (e.g. $c=0$).
 
\item[(iii)] A quandle $(Q,\cdot)$ is a $S$-right quasigroup if and only if $$x=((xy)x)y$$ 
 holds. Let $G$ be a group and $Q$ be the conjugation quande associated to $G$. Then $Q$ is a $S$-right quasigroup if and only if $$[x,y^2]=1$$ 
 holds, i.e. $\setof{y^2}{y\in G}\subseteq Z(G)$ ($G$ is also said to be $2$-central). 
 
 \item[(iv)] A group $(G,\cdot)$ is a right quasigroup and $x/y=xy^{-1}$. Therefore, $(G,\cdot)$ is a $\S$-right quasigroup if and only if
 $$xy^{-1}=x(yx)\, \Leftrightarrow \, x=y^{-2}$$
holds. Such an identity is satisfied only by the trivial group.
\end{itemize}
\end{example} 

\subsection{Regular and Singular links, quandles and singquandles}

An oriented  link in $S^3$ is an embedding $S^1\vee S^1\cdots \vee S^1$ of $\mu$ disjoint copies of $S^1$ in $ S^3$, together with the choice of an orientation on every connected component. Each oriented link can be represented by means of a regular diagram on a plane, that is a plane quadrivalent directed graph having vertices decorated with overcrossing and undercrossing (see Figure \ref{fig:crossing}). 
\begin{figure}[ht]
\begin{center}
\includegraphics[width=6cm]{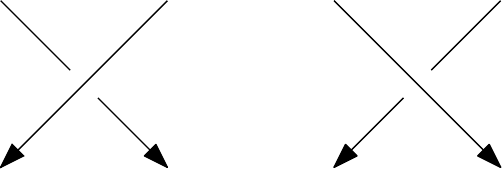}
\caption{The two possible decoration of a vertex in a digram of a link.}
\label{fig:crossing}
\end{center}
\end{figure}
As oriented links in $S^3$ are considered up to isotopy, their diagrams are considered up to planar isotopy (preserving decorations) and the classical  Reidemeister moves depicted  in Figure \ref{fig:reidemeister}.

\begin{figure}[ht]
\begin{center}
\includegraphics[width=8cm]{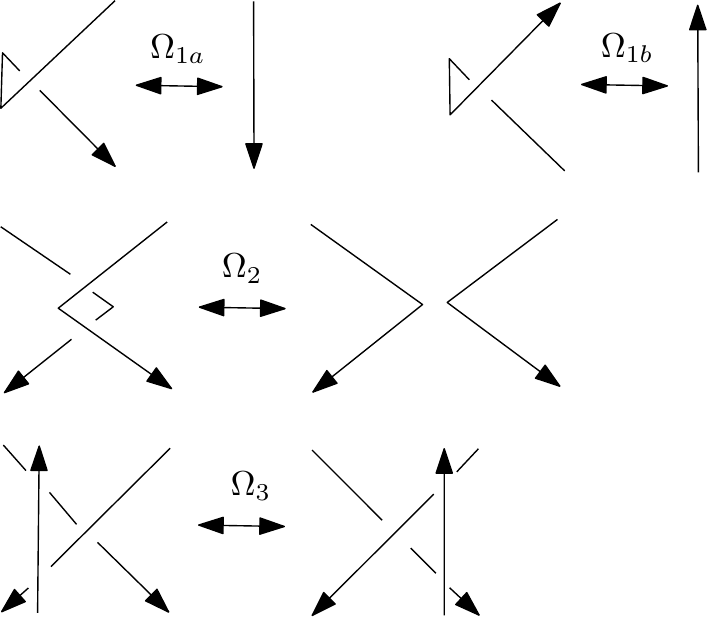}
\caption{The classical Reidemeister moves.}
\label{fig:reidemeister}
\end{center}
\end{figure}

An usual  way to construct link invariants is to define them on diagrams in a way that ensure invariance under Reidemeister moves. With respect to this approach,  algebraic invariant are provided by quandles. Indeed, the  axiom satisfied by a quandle structure, idempotency, right-quasigroup property, right-distributivity, correspond exactly to the  Reidemeister moves. More precisely given an non-empty set $Q$ and a binary operation $\ast$, suppose to decorate each arc of a link diagram by an element of $Q$ as  depicted in Figure \ref{fig:crossing_quandle}: requiring the invariance under the Reidemeister moves, correspond to imposing  conditions giving a quandle structure to the binary operation. So, to each oriented link $L$ we can associate a quandle $Q(L)$, called  \textit{fundamental quandle} of $L$, by taking the  quotient of the free quandle generated by the arcs  of a diagram of $L$ modulo the crossing relations represented in Figure \ref{fig:crossing_quandle}. A coloring of a link by a quandle $T$ is  a quandle homomorphism $f\colon Q(L)\to T$ or equivalently a decoration of the edges of a link diagram with elements of $T$ such that the crossing relations in Figure \ref{fig:crossing_quandle} are satisfied; so a quandle could color a link $L$  if and only if it is a quotient of the fundamental quandle of $L$.

\begin{remark}\label{remark on coloring} Given a coloring of a link $L$ by a quandle $T$,  the colors (i.e. the elements of $T$) that decorate the edges within the same connected component of $L$ are also within the same connected component of $T$. Moreover, for a knot any pair of colors at an arbitrary crossing completely determines the whole coloration. So, all quandles that color a knot diagram are 2-generated and connected according to Lemma \ref{connected by gen}.
\end{remark}

In \cite{Fenn},  the fundamental quandle was introduced in a  topological way using paths in the link complement. Using this approach it is easy to see to see that the fundamental quandle in fact depends on the orientation of the couple $(S^3,L)$: if both of them are changed the quandle structure does not change, while if only one of them is changed the quandle $(Q(L),\cdot,/)$ and $(Q(L'),/,\cdot)$ are isomorphic where $Q(L')$ is the fundamental quandle associated to the pair with just one orientation reversed.

\begin{figure}[ht]
\begin{center}
\includegraphics[width=6cm]{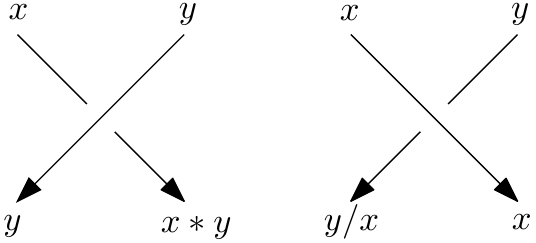}
\caption{The coloring of crossing in the fundamental quandle.}
\label{fig:crossing_quandle}
\end{center}
\end{figure}

The above way of reasoning could be generalized to the case of oriented singular links. An oriented singular link is the immersion  $S^1\vee S^1\cdots \vee S^1$ of $\mu$ disjoint copies of $S^1$ in $ S^3$, together with the choice of an orientation on every connected component.  Taking a  combinatorial point of view, an equivalent definition of a  singular link is as an equivalence class of plane quadrivalent directed graph having some vertices decorated with overcrossing and undercrossing, modulo the equivalence relation  generated by  planar isotopy (preserving decorations) and the  Reidemeister moves depicted  in Figure \ref{fig:reidemeister} and in Figure \ref{fig:reidemeister2} (see \cite{Generating_Reidemeister}). 

\begin{figure}[ht]
\begin{center}
\includegraphics[width=14cm]{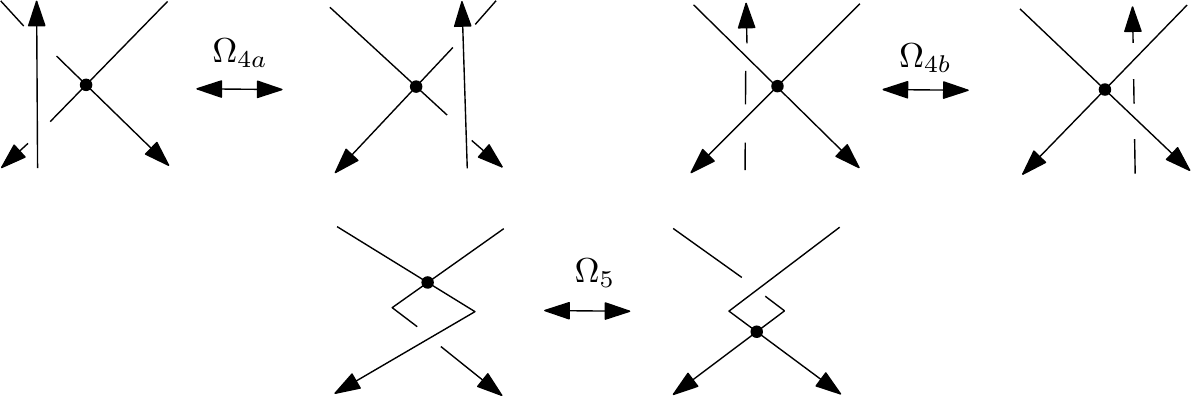}
\caption{The  Reidemeister moves for singular crossings.}
\label{fig:reidemeister2}
\end{center}
\end{figure}

\begin{figure}[ht]
\begin{center}
\includegraphics[width=10cm]{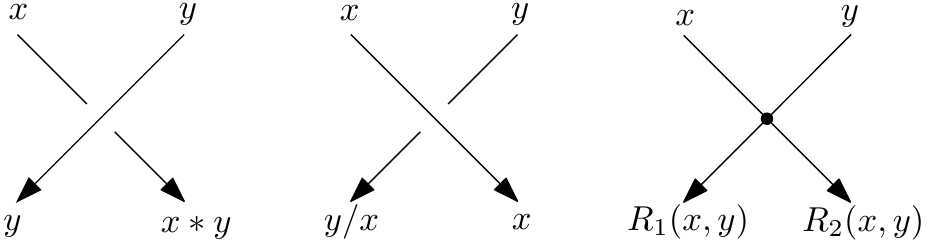}
\caption{The coloring of crossing of a singular diagram.}
\label{fig:crossing_sing_quandle}
\end{center}
\end{figure}

The decorated crossings are called regular crossing, and the others are called singular crossing. The set of moves displayed in Figures \ref{fig:reidemeister} and \ref{fig:reidemeister2} are called  singular Reidemeister moves. In \cite{Generating_Reidemeister} an algebraic structure having three binary operations $(\ast, R_1,R_2)$, associated to a singular link and generalizing  the quandle one, is constructed. It is called  the fundamental oriented singquandle and it is defined as follows.

Suppose to color the arcs of the diagram of a singular links as in Figure \ref{fig:crossing_sing_quandle}; as before imposing the invariance with respect to the classical Reidemeister moves implies that  the operation $*$ is a quandle, while the  invariance under the other generalized Reidemeister moves, imposes some further axioms (see Figure \ref{fig:esempio_reidemeister} for $\Omega_5$). After a suitable change of variables we can rewrite them as 
\begin{align}
R_1(x,y)*z &=R_1(x*z,y*z) \label{O1}	\tag{OS1}\\
R_2(x,y)*z &=R_2(x*z,y*z) \label{O2}\tag{OS2}\\
(y*x)*z &=(y*R_1(x,z))*R_2(x,z) \label{O3}\tag{OS3}\\
R_1(x,y)*R_2(x,y) &=R_2(y,x*y) \label{O5}\tag{OS4}.\\
R_2(x,y) &=R_1(y,x*y) \label{O4}\tag{OS5}
\end{align}

\begin{figure}[ht]
\begin{center}
\includegraphics[scale=1.1]{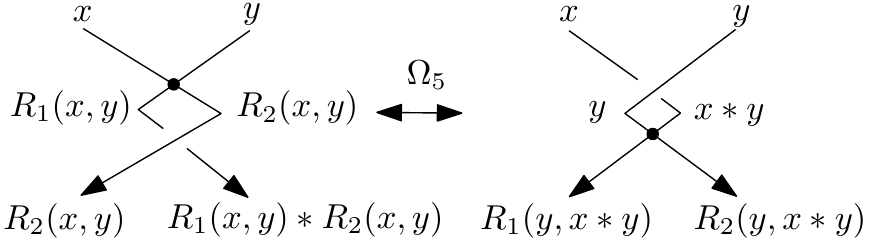}
\caption{The axioms associated to $\Omega_5$.}
\label{fig:esempio_reidemeister}
\end{center}
\end{figure}

So we get the following definition.

\begin{definition}\label{Oriented def}\cite{Generating_Reidemeister}
An {\it oriented singquandle} is a triple $(X,*,R_1,R_2)$ where $(X,*)$ is a (right) quandle and
$$R_1, \, R_2:X\times X\longrightarrow X$$
such that \eqref{O1}, \eqref{O2}, \eqref{O3}, \eqref{O5}, \eqref{O4} hold.

\end{definition}

In \cite{OrientedSing}, the fundamental oriented singquandle $SQ(L)$ associated to a singular link $L$ is defined as the quotient of the free singquandle  generated by the arcs  of any diagram of $L$ modulo the crossing relations represented in Figure \ref{fig:crossing_sing_quandle}. As for classical knots colorings of a link $L$ by an oriented singquandle $T$ correspond to morphisms from $SQ(L)$ onto $T$.  

Let us see what is going to happen if we forget about orientation: the quandle structure will be involutive and the other two binary operations will have to respect a rotational simmetry. More precisely, following \cite{SingInvol}, we get the following definition.

\begin{definition}\cite{SingInvol}\label{def singquandles}
 Let ($X,*)$ be an involutive quandle and let $R_1,R_2:X\times X\longrightarrow X$. The triple $(X,*,R_1,R_2)$ is called a singquandle if the following axioms hold:
\begin{align} 
 x &= R_1(y,R_2(x,y)) = R_2(R_2(x, y),R_1(x, y))\label{D1} \tag{S1a}\\
 y &= R_2(R_1(x, y), x) = R_1(R_2(x, y),R_1(x, y))\label{D2}\tag{S1b}\\
 (R_1(x, y),R_2(x, y)) &= (R_2(y,R_2(x, y)),R_1(R_1(x, y), x))\label{def of R2}\tag{S1c}\\
(y * z) * R_2(x, z) &= (y * x)* R_1(x, z) \tag{S2}\\
R_1(x, y) &= R_2(y * x,x)\tag{S3}\\
 R_2(x, y) &= R_1(y * x, x) * R_2(y * x, x) \label{TO CHECK}\tag{S4}\\
 R_1(x * y, z) * y &= R_1(x, z * y) \tag{S5}\\
 R_2(x * y, z) &= R_2(x, z * y) * y \tag{S6}
\end{align}
\end{definition}

Notice that the axioms (\ref{D1}), (\ref{D2}), (\ref{def of R2}) are those corresponding to  a rotational simmetry of $\pi/2$, $\pi$ and $3/2\pi$ of  the coloration of a singular crossing (see Figure \ref{fig:fare});  hence it is enough to set the symmetry with respect to a rotation of $\pi/2$ degree and so set only the axiom 
\begin{align}
\left(R_2(x,y), y\right)=\left(R_1(R_1(x,y),x),R_2(R_1(x,y),x)\right)\label{rotational}\tag{S1}.
\end{align}

\begin{figure}[ht]
\begin{center}
\includegraphics[width=16cm]{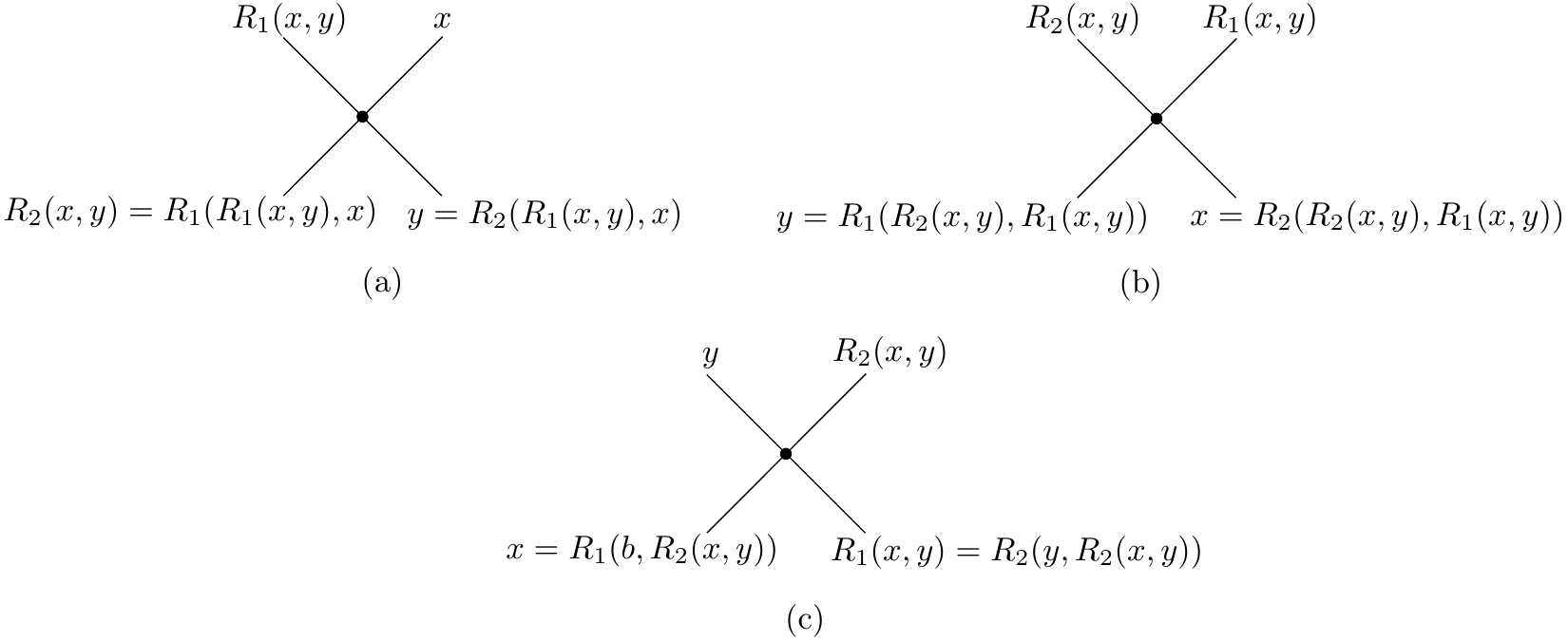}
\caption{Rotational simmetry and corresponding singquandle axioms.}
\label{fig:fare}
\end{center}
\end{figure}

The fundamental singquandle associated to a (non-oriented) singular link is defined analogously to the definition of the oriented singquandle.

\section{Oriented singquandles}

\label{oriented}

If we set $y x=R_1(x,y)$ and we take \eqref{O4} as the definition of $R_2$ by $\cdot$ and $*$ as $R_2(x,y)=(x*y)y$,  we can rewrite  Definition \ref{Oriented def} as follows. 

\begin{definition}\label{Oriented def2}
An {\it oriented singquandle} is a binary algebraic structure $(X,\cdot,*)$ where $(X,*)$ is a (right) quandle and
such that
\begin{align}
(yx)*z &= (y*z)(x*z) \label{O21}\tag{OS1'}	\\
(x(yx))*z &=(x*z)((y*z)(x*z)) \label{O22}\tag{OS2'}\\
(y*x)*z &=(y*(zx))*(x(zx)) \label{O23}\tag{OS3'}\\
(yx)*((x*y)y) &=(y*(x*y))(x*y) \label{O25}\tag{OS4'}
%
\end{align}
\end{definition}

It is clear that \eqref{O21} implies \eqref{O22}, so we can omit \eqref{O22} from the definition of oriented singquandles.

\begin{proposition}\label{Car Oriented}
Let $(X,\cdot,*)$ be a binary algebraic structure. The following are equivalent:
\begin{itemize}
\item[(i)] $(X,\cdot,*)$ is an oriented singquandle.
\item[(ii)] $(X,*)$ is a quandle, $\rho_x\in \aut{X,\cdot}$ and
\begin{align}\label{prop of rho}
\rho_x \rho_y	&=\rho_{(y*x)x}\rho_{xy}.
\end{align}
for every $x,y\in X$.
\item[(iii)] $(X,*)$ is a quandle and the following identities hold:
\begin{align}
(xy)*z&=(x*z)(y*z),\label{Oriented1}\\
(z*y)*x &=(z*(xy))*((y*x)x)\label{Oriented2}.
\end{align}

\end{itemize} 
\end{proposition}

\begin{proof}

The equivalence between (ii) and (iii) is straightforward. Let us show the equivalence between (i) and (ii).

\begin{itemize}
\item The identity \eqref{O21} is equivalent to
have that $\rho_z\in \aut{X,\cdot}$ for every $z\in X$. In particular, $R_z$ and $\rho_z$ commute: indeed $\rho_z R_z(x)=(xz)*z=(x*z)(z*z)=(x*z)z=R_z\rho_z(x)$ for every $x\in X$.


\item The identity \eqref{O23} is equivalent to
\begin{equation}\label{e1}
\rho_x \rho_y	=\rho_{(y*x)x}\rho_{xy}.
\end{equation}

\item  Using that $\rho_x\in \aut{(X,\cdot)}$ and $\rho_{x*y}=\rho_y \rho_x\rho_y^{-1}$, the identity \eqref{O25} can be written as

$$\rho_{(x*y)y}R_x(y)=R_{x*y}\rho_{x*y}(y)=\rho_y R_x \rho_y^{-1} \rho_y \rho_x \rho_y^{-1}(y)=\rho_y R_x\rho_x(y).$$
Then using \eqref{e1} we have $\rho_{(x*y)y}=\rho_y \rho_x\rho_{yx}^{-1}$ and so

$$\rho_{(x*y)y}R_x(y)=\rho_y \rho_x\rho_{yx}^{-1}(yx)=\rho_y \rho_x R_x(y)= \rho_y R_x\rho_x(y)$$
holds since $\rho_x$ and $R_x$ commute. 
\end{itemize}
\end{proof}

Let $(X,\cdot,*)$ be a binary algebraic structure. 
\begin{itemize}
\item[(i)] If $(X,*)$ is projection then $(X,\cdot,*)$ is an oriented singquandle (the correponding maps are $R_1(x,y)=yx$ and $R_2(x,y)=xy$). Therefore, given any binary structure $(X,\cdot)$ we can color the diagram of a singular knot as  as in Figure \ref{fig:es1}.

\begin{figure}[ht]
\begin{center}
\includegraphics[width=10cm]{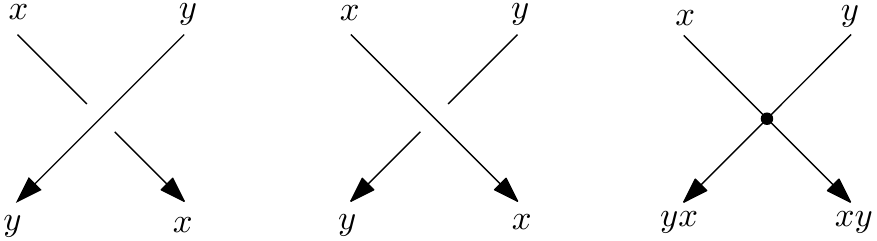}
\caption{}
\label{fig:es1}
\end{center}
\end{figure}

\item[(ii)] If $(X,\cdot)$ is a projection right quasigroup then $\rho_x\in \aut{X,\cdot}=\sym{X}$ for every $x\in X$ and \eqref{prop of rho} turns out to be right distributivity for $(X,*)$ (the correponding maps are $R_1(x,y)=y$ and $R_2(x,y)=x*y$). Hence $(X,\cdot,*)$ is an oriented singquandle for any quandle $(X,*)$ and we can color the diagram of a singular knot as in Figure \ref{fig:es2}.

\begin{figure}[ht]
\begin{center}
\includegraphics[width=10cm]{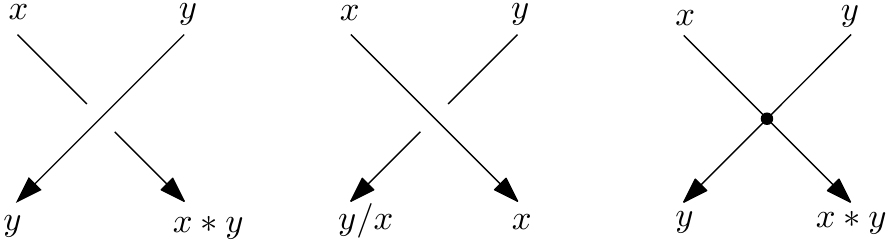}
\caption{}
\label{fig:es2}
\end{center}
\end{figure}

Moreover, if $\bar L_+$ is the link obtained by replacing each singular crossing with the regular crossing on the left  of Figure \ref{fig:es2}, then $Q(\bar L_+)$ is a quotient of $SQ(L)$. Indeed the link $L$ is colorable by the oriented singquandle $X$ generated by the arcs modulo the crossing relations as in \eqref{fig:es2} that is isomorphic to $Q(\bar{L}_+)$ (the $\cdot$ operation in $X$ is trivial and the generators of $X$ satisfy the very same relations satisfied by the generators of $Q(\bar{L}_+)$). Therefore we have the canonical morphism 
$$SQ(L)\longrightarrow X\cong Q(\bar{L}_+)$$
that identifies the canonical generators.

\item[(iii)] Let $(X,*,/)$ be a quandle. Then $(X,/,*)$ is an oriented singquandle. Indeed clearly $\rho_x\in \aut{X,/}=\aut{X,*}$ for every $x\in X$ and
\begin{align*}
(y*(z/x))*((x*z)/z)&=(y*(z/x))*x\\
&=(y*x)*((z/x)*x)\\
&=(y*x)*z, 
\end{align*}
i.e. \eqref{Oriented2} holds. In this case the crossing relations look like in Figure \ref{fig:es3}.
\begin{figure}[ht]
\begin{center}
\includegraphics[width=10cm]{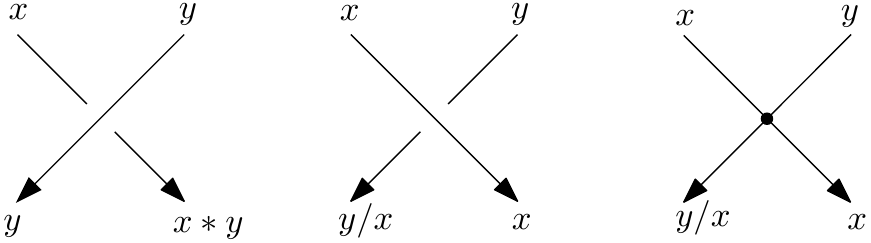}
\caption{}
\label{fig:es3}
\end{center}
\end{figure}

Similarly to the previous case, if $\bar L_-$ is the link obtained by replacing each singular crossing with the regular crossing on the right of Figure \ref{fig:es2}, then $Q(\bar L_-)$ is a quotient of $SQ(L)$.

\end{itemize}
The axioms in Proposition \ref{Car Oriented}(iii) are independent, as we can see from the following computer generated examples, computed by using Mace 4 \cite{Prover9}. 
\begin{itemize}
\item The binary algebraic structure $(X,\cdot,*)$ where
$$(X,\cdot)=\begin{tabular}{|c c c|}
\hline
1&1&1\\
1&2&3\\
3& 3& 3\\
\hline

\end{tabular}\,,\qquad (X,*)=\begin{tabular}{|c c c |}
\hline
1&3&1\\
2&2&2\\
3&1&3\\
\hline
\end{tabular}\,,
$$
satisfies \eqref{Oriented1} but not \eqref{Oriented2}.
\item The binary algebraic structure $(X,\cdot,*)$ where
$$(X,\cdot)=\begin{tabular}{|c c c|}
\hline
2&1&1\\
2&1&1\\
2&1&1\\
\hline
\end{tabular}\,,\qquad (X,*)=\begin{tabular}{|c c c|}
\hline
1&1&1\\
3&2&2\\
2&3&3\\
\hline
\end{tabular}\,,
$$
satisfies \eqref{Oriented2} but not \eqref{Oriented1}.
\end{itemize}

\section{Singquandles}
\label{non-oriented}

Let $(X,*,R_1,R_2)$ be a singquandle and let us define $y x=R_1(x,y)$ and according to \eqref{rotational}, $R_2(x,y)=R_1(R_1(x, y), x)=x(yx)$ (see Figure \ref{sing coloring}). Let us show Definition \ref{def singquandles} in term of identities satisfied by $\cdot$ and $*$.
\begin{align} 
y &=(yx)(x(yx))\tag{S1'}	\label{eq to S}\\
(y * z) * (x(zx)) &= (y * x)*(zx)\label{prop of rho2}\tag{S2'} \\
yx &= (y * x)(x(y*x))\label{def of *}\tag{S3'}\\
 x(yx) &= (x(y * x)) * ((y * x)( x(y*x))\label{last one}\tag{S4'}\\
(z(x * y)) * y &= ( z * y)x \label{auto}\tag{S5'}\\ 
 (x * y) (z(x*y)) &= (x ((z*y)x))*y \label{auto2}\tag{S6'}
\end{align}

\begin{proposition}\label{caratt Sing}
Let $(X,\cdot,*)$ be a binary algebraic structure. The following are equivalent:
\begin{itemize}
\item[(i)] $(X,\cdot,*)$ is a singquandle.
\item[(ii)] $(X,\cdot)$ is a $\S$-right quasigroup, $x*y=(xy) y$ and $(X,\cdot,*)$ is an oriented singquandle.
\item[(iii)] 
The following identities hold
\begin{align}
(x(yx))y&=x\label{S again}\\
x*y&=(xy)y\label{def of quandle}\\
x*x &= x\label{Idempotency_2}\\
((xy)z)z &= ((xz)z)((yz)z)\label{distributive_2}\\
(x*z)*y &=(x*(yz))*(z/y) \label{ID3_2}
\end{align}
\end{itemize}
\end{proposition}

\begin{proof}
Let us first point out some observations:
\begin{itemize}
\item According to Proposition \ref{PropS0}, the identity \eqref{eq to S} is equivalent to have that $(X,\cdot)$ is a $\S$-right quasigroup. 


\item The identity \eqref{prop of rho2} is equivalent to
\begin{equation}\label{rho}
\rho_{zx}\rho_{x(zx)}=\rho_x\rho_z.
\end{equation}
\end{itemize}

(i) $\Rightarrow$ (ii) Using $\rho_x^2=1$ and replacing $x$ by $x*y$ in \eqref{auto} we have 
$$(z((x * y)*y)) * y =(zx)*y= ( z * y)(x*y) $$
namely $\rho_x\in \aut{X,\cdot}$. In particular, $R_x$ and $\rho_x$ commute.

Using that $\rho_x\in \aut{X,\cdot}$ in \eqref{def of *} we have
$$yx = (y * x)(x(y*x))= (y * x)((x*x)(y*x))=(y(xy))*x=(y/x)*x.$$
Thus, $R_x=\rho_x R_x^{-1}$, i.e. $\rho_x=\rho_x^{-1}=R_x^{2}$ and so $x*y=(xy)y$. 

Using that $(x*z)z=x/z=x(zx)$ in \eqref{rho} we have that \eqref{prop of rho} follows.

Therefore $(X,\cdot,*)$ is an oriented singquandle according to Proposition \ref{Car Oriented}.

(ii) $\Rightarrow$ (iii) The mapping $\rho_x^{2}\in \aut{X,\cdot}$ and so \eqref{distributive_2} holds. $(X,*)$ is a quandle and so $x*x=(xx)x=x$. 
Finally \eqref{ID3_2} follows by \eqref{prop of rho} just by replacing the definition of $*$.

(iii) $\Rightarrow$ (i) Let us first show that $(X,*)$ is an involutory quandle. Since $\rho_x\in \aut{X,\cdot}$ then $\rho_x\in \aut{X,*}$ and $(X,*)$ is idempotent by \eqref{Idempotency_2}. Thus $(X,*)$ is a quandle. 


Note that 
\begin{align}
    (x/y)(yx)&\overset{\eqref{reduced axioms}}{=}(x(yx))(yx)=x*(yx)\label{A},\\
x/(yx)&\overset{\eqref{reduced axioms}}{=}x((yx)x)=x(y*x)\overset{\eqref{Idempotency_2}}{=}(x*x)(y*x)\overset{\eqref{distributive_2}}{=}(xy)*x\label{B}.
\end{align}
Therefore we have
\begin{align}
    (x*(yx))*(x/y)\overset{\eqref{A}}{=} ((x/y)(yx))*(x/y)\overset{\eqref{B}}{=}(x/y)/((yx)(x/y)\overset{\eqref{reduced axioms}}{=} (x/y)/y.\label{C}
\end{align}
Finally, we have
\begin{align*}
    x*y\overset{\eqref{Idempotency_2}}{=}(x*x)*y\overset{\eqref{ID3_2}}{=}(x*(yx))*(x/y)\overset{\eqref{C}}{=}(x/y)/y.
\end{align*}
Therefore $\rho_y=R_y^2=R_y^{-2}=\rho_y^{-1}$, i.e. $(X,*)$ is involutory.

Under this assumption \eqref{auto} is equivalent to $\rho_y\in \aut{X,\cdot}$ and \eqref{auto2} follows from the same argument. Also \eqref{def of *} follows as in the first part of the proof, by using that $\rho_x$ is an involutory automorphism of $(X,\cdot)$. The identity \eqref{prop of rho2} is equivalent to \eqref{distributive_2} modulo the identity \eqref{def of quandle}.

Let us check the identity \eqref{last one}. Indeed, using that $y*x=(yx)x$ and that $x/y=x(yx)$ we have
\begin{align*}
(x(y * x)) * ((y * x)( x(y*x))&=(x((y x)x)) * ((y * x)/ x)\\
&=(x/(yx)) * (yx)=x(yx).
%
\end{align*}

\end{proof}

Singquandles are $2$-divisible.

\begin{corollary}\label{sing are 2 div}
Singquandles are $2$-divisible and the squaring mapping is an involution.
\end{corollary}

\begin{proof}

Let $(X,\cdot)$ be a singquandle and $x\in X$. It is enough to prove that $\sigma^2(x)=(xx)(xx)=x$. According to \eqref{Idempotency_2} we have $xx=x/x$. Hence by \eqref{reduced axioms} it follows that
\begin{equation*}
(xx)(xx)\overset{\eqref{Idempotency_2}}{=}(xx)(x/x)\overset{\eqref{reduced axioms}}{=}x.
\end{equation*}

\end{proof}
Note that the proof of Corollary \ref{sing are 2 div} actually uses just \eqref{reduced axioms} and \eqref{Idempotency_2}, so also $\S$-right quasigroups such that $(xx)x=x$ holds are $2$-divisible.

The set of axioms given in Proposition \ref{caratt Sing}(iii) is independent. We can consider singquandles as binary structure with one binary operation using \eqref{def of quandle} as the definition of $*$ and rewrite all the axioms accordingly as:
\begin{align}
(x(yx))y&=x,\label{S again3}\\
(xx)x &= x,\label{Idempotency_3}\\
((xy)z)z &= ((xz)z)((yz)z),\label{distributive_3}\\
(((xz)z)y)y &=(((x(yz))(yz))(z/y))(z/y). \label{ID3_3}
\end{align}

Let us show that the axioms above are independent by examples generated by the software Mace4:
\begin{itemize}
\item Involutory $\S$-right quasigroups are singquandles. Indeed, if $(X,\cdot)$ is involutory then $(X,*)$ is projection and so \eqref{Idempotency_2}, \eqref{distributive_2} and \eqref{ID3_2} are trivially satisfied.
The involutory right quasigroup
$$(X,\cdot)=\begin{tabular}{|c c |}
\hline
2&1\\
1&2\\
\hline
\end{tabular}\,,
$$
does not satisfy \eqref{S again3}.

\item The right quasigroup 
$$(X,\cdot)=\begin{tabular}{|c c c c|}
\hline
2&3&3&2\\
4&1&1&4\\
1&4&4&1\\
3&2&2&3\\
\hline
\end{tabular}\,,
$$
satisfies all the axioms but \eqref{Idempotency_3}.

\item The right quasigroup
$$(X,\cdot)=\begin{tabular}{|c c c c c c|}
\hline
1&4&6&1&1&1\\
3&2&2&3&3&3\\
2&3&3&2&2&2\\
4&1&5&5&5&4\\
5&6&4&4&4&5\\
6&5&1&6&6&6\\
\hline
\end{tabular}\,,
$$
 satisfies all the axioms but \eqref{distributive_3}.

\item The right quasigroup 
$$(X,\cdot)=\begin{tabular}{|c c c c c|}
\hline
1&4&5&1&1\\
3&2&2&3&3\\
2&3&3&2&2\\
4&5&1&4&4\\
5&1&4&5&5\\
\hline
\end{tabular}\,,
$$
satisfies all the axioms but \eqref{ID3_3}.

\end{itemize}

\begin{corollary}
Let $(X,\cdot,/)$ be a singquandle. Then $(X,/,\cdot)$ is a singquandle.
\end{corollary}

\begin{proof}
According to Proposition \ref{S and S'}, $(X,/,\cdot)$ is a $\S$-right quasigroup and so \eqref{S again3} hold for $(X,/,\cdot)$. Note that we can write \eqref{distributive_3} and \eqref{ID3_3} in terms of the right multiplication mappings as
\begin{align*}
    \eqref{distributive_3}\,&\Leftrightarrow\, R_{R_z^2(y)}R_z^{2} =R_z^2 R_y, \\
\eqref{ID3_3}\,&\Leftrightarrow\,  R_y^2 R_z^2=R_{z/y}^2 R_{yz}^2.
\end{align*}
Thus, using that $R_z^4=1$, we have
\begin{align*}
 (xx)x=x\, & \Leftrightarrow \, (x/x)/x=x,\\
 R_{R_z^2(y)}R_z^{2} =R_z^2 R_y \, & \Leftrightarrow \, R_{R_z^{-2}(y)}^{-1}R_z^{-2} =R_z^{-2} R_y^{-1}\\ 
R_y^2 R_z^2=R_{z/y}^2 R_{yz}^2\,& \Leftrightarrow\,R_z^{-2} R_y^{-2}=R_{yz}^{-2} R_{z/y}^{-2},
\end{align*}
namely all the axioms of singquandles hold for $(X,/,\cdot)$.
\end{proof}

Finally, note that colorings of non-oriented singular links by singquandles are obtained as in Figure \ref{sing coloring}.

\begin{figure}[ht]
\begin{center}
\includegraphics[width=10cm]{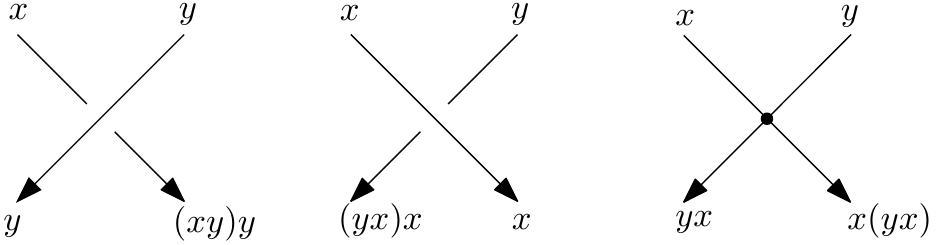}
\caption{}
\label{sing coloring}
\end{center}
\end{figure}

Note that, since singquandles are right quasigroups, Remark \ref{remark on coloring} holds also for such colorings.

	\subsection{Affine singquandles}
Let us turn our attention to the family of affine singquandles.
\begin{proposition}\label{affine Singquandles}
Let $(X,\cdot)=\aff{(A,f,g,c)}$ be an affine right quasigroup. The following are equivalent:
\begin{itemize}
\item[(i)] $(X,\cdot)$ is an singquandle.

\item[(ii)] The identities \eqref{reduced axioms}, \eqref{Idempotency_3} and $(((xy)y)y)y=x$ hold. 

\item[(iii)] The following identities hold:
\begin{align}
fg^2+f^2-1  &=0, \label{aff_1}\tag{A1}\\
g+fgf&=0,\label{aff_2}\tag{A2}\\
1-f^4&=0\label{aff_3}\tag{A3}\\
(1+f)g&=1-f^2\label{aff_4}\tag{A4}\\
(1+f)(c)&=g(c)=0\label{aff_5}\tag{A5}.
\end{align}

\end{itemize}


\end{proposition}


\begin{proof}
(i) $\Rightarrow$ (ii) According to Proposition \ref{caratt Sing}, \eqref{reduced axioms} and \eqref{distributive_3} hold and $(X,*)$ where $x*y=(x y) y$ is an involutory quandle. Then $(x*y)*y=(((xy)y)y)y=x$ holds.

(ii) $\Rightarrow$ (iii) We compute the conditions on $f,g$ and $c$ that need to be satisfied in order to have that \eqref{reduced axioms}, \eqref{Idempotency_3} and $(((	xy)y)y)y=x$ hold. 

\begin{itemize}
\item The identity \eqref{reduced axioms} holds if and only if the identities \eqref{affine S} hold, i.e. \eqref{aff_1}, \eqref{aff_2} and $(1+fg+f)(c)=0$. hold.

\item The identity \eqref{Idempotency_3} holds if and only if
\begin{align*}
(xx)x&=c+g(x)+f(c)+fg(x)+f^2(x)=x,
\end{align*}
namely
\begin{align}
(1+f)(c)=0,\, (1+f)g=1-f^2. \label{from ID}
\end{align}
Using the first equation of \eqref{from ID} and that $(1+fg+f)(c)=0$ we have that $(1+f)(c)=g(c)=0$. Therefore, \eqref{aff_4} and \eqref{aff_5} hold.

\item We have that $((xy)y)y)y=x 
$ if and only if 
\begin{align*}
(((x y) y) y) y &=(1+f^2)(1+f)(c)+(1+f^2)(1+f)g(y)+f^4(x)\\
&\overset{\eqref{aff_5}}{=}(1+f^2)(1+f)g(y)+f^4(x)\\
&\overset{\eqref{aff_4}}{=}(1+f^2)(1-f^2)(y)+f^4(x)\\
&=(1-f^4)(y)+f^4(x)=x
\end{align*} 
i.e. \eqref{aff_3} holds.
\end{itemize}

(iii) $\Rightarrow$ (i) Let us prove that such conditions are sufficient for the other axioms of $S$-quandles. We have already showed that \eqref{reduced axioms} and \eqref{Idempotency_3} are equivalent to the equations \eqref{aff_1}, \eqref{aff_2}, \eqref{aff_3}, \eqref{aff_4} and \eqref{aff_5}.
Let us check the other identities.

\begin{itemize}
\item  Since 
\begin{align*}
((xy)z)z &=f^2(c)+(g+fg)(z)+f^2g(y)+f^3(x)\\
((xz)z)((yz)z)&=f^2(c)+(g^2+gfg+fg+f^2g)(z)+(gf^2)(y)+f^3(x),
\end{align*}
the identity \eqref{distributive_3} holds if and only if 
\begin{align}
g&=g^2+gfg+f^2g,\label{aff4}\\
f^2g&=gf^2 \label{aff5}
\end{align}
Since $f^4=1$ then by \eqref{aff_2} we have
$$f^2g=-fgf^{-1}=gf^{-2}=gf^2.$$
Moreover
\begin{align*}
g^2+gfg+f^2g-g&\overset{\eqref{aff5}}{=}g^2+gfg+gf^2-g=g(g+fg+f^2-1) \overset{\eqref{aff_4}}{=}0.
\end{align*}

\item 

It is easy to compute that
$$(((xz)z)y)y=(((xz)z)y)y=(1-f^2)(y-z)+x.$$
So we have 
\begin{align*}
(((x(yz))(yz))(z/y))(z/y)&=(1-f^2)(z/y-yz)+x\\
&=(1-f^2)(f^{-1}(z-g(y)-c)-c-g(z)-f(y))+x\\
&\overset{\eqref{aff_5}}{=}(1-f^2)((f^{-1}-g)(z)-(f^{-1}g+f)(y))+x.
\end{align*}
Thus, the identity \eqref{ID3_3} holds if and only if 
\begin{align*}
(1-f)(1+f)(1+f^{-1}g+f)&=0\\
(1-f)(1+f)(f^{-1}-g+1)&=0
\end{align*}
Using that $(1+f)g=1-f^2$ and that $1-f^4=0$ we have
\begin{align*}
(1-f)(1+f)(1+f^{-1}g+f)&=(1-f)(1+f+f^{3}(1-f^2)+f+f^2)\\
&=(1-f)(1+f+f^2+f^3)=1-f^4=0\\
(1-f)(1+f)(f^{-1}-g+1)&=(1-f)(f^3+1-(1-f^2)+1+f)\\
&=(1-f)(1+f+f^2+f^3)=1-f^4=0.
\end{align*}
Thus the identity \eqref{ID3_2} holds.
\end{itemize}
\end{proof}
The construction of Alexander singquandles given in \cite[Proposition 4.3]{SingInvol} defined as a binary algebraic structure over an abelian group $A$ using $t,B\in \aut{A,+}$ by setting 
$$x * y = tx + (1 - t)y,\qquad R_1(x,y) = (1+ t - B)x + (t + B)y,\qquad R_2(x,y) = (1 -B)x + By,$$
provides exactly affine idempotent Sinquandles (the relation between $f, g$ and $B$ is $g=B(1-B)$ and $f=B^2-B+1$ and the pair $B$ and $t$ satisfy $1-(1-B)^4=B(1+(1-B)^2)=(1-B)^2-t=0$).

It is easy to check that the affine right quasigroup $\aff(A,f,g,c)$ is a idempotent if and only if
\begin{align}\label{idempotent}
g =1-f,\quad  c=0.
\end{align}
Note that idempotent affine right quasigroup are quandles and so \eqref{distributive_2} holds and according to Example \ref{Example of S}(iii), the identity \eqref{reduced axioms} is equivalent to the identity $((xy)x)y=x$. Thus, under the assumptions \eqref{idempotent} and using that $1-f^4=(1+f)(1-f)(1+f^2)$ the identities in Proposition \ref{affine Singquandles}(iii) reduce to

$$(1-f)(1+f^2)=0.$$
So, we have the following result.

\begin{corollary}\label{idempotent affine}
Let $X=\aff(A,f,g,c)$ be an affine right quasigroup. The following are equivalent:
\begin{itemize}
\item[(i)] $X$ is an idempotent Sinquandle.
\item[(ii)] The identities $xx=x$ and $((xy)x)y=x$ hold.
\item[(iii)] $g=1-f$, $(1-f)(1+f^2)=0$ and $c=0$.
\end{itemize}
\end{corollary}
According to Corollary \ref{idempotent affine}, affine idempotent Singquanddles are endowed with a module structure over the ring $R=\mathbb{Z}[t,t^{-1}]/((1-t)(1+t^2))$. Conversely, every module $M$ over $R$ is an idempotent affine singquandle with the operation
$$x\cdot y=(1-t)x+ty$$
for $x,y\in M$. In particular, given an affine quandle $Q=\aff(A,1-f,f,0)$ we can consider the right quasigroup $Q'=\aff(A/\left((1-f)(1+f^2)A\right),1-f',f',0)$ where $f'$ is the automorphism induced by $f$ on the quotient group $A/\left((1-f)(1+f^2)A\right)$. Then $Q'$ is a singquandle.

\section*{Acknowledgements}
The first author would like to thank M. Kynion for his help with Prover9. The second author  has been supported by the ``National Group for Algebraic and Geometric Structures, and their Applications" (GNSAGA-INdAM) and University of Bologna, funds for selected research topics.

\bibliographystyle{amsalpha}
\bibliography{references} 

\end{document}